\documentclass[twoside,12pt,reqno]{amsart}

\usepackage{latexsym}

\newcommand{\qdn}{\hspace*{-1.5mm}}
\newcommand{\qqdn}{\hspace*{-2.5mm}}
\newcommand{\xqdn}{\hspace*{-5.0mm}}
\newcommand{\xxqdn}{\hspace*{-10mm}}

\newcommand{\fns}{\footnotesize}

\newcommand{\ffnk}[4]{\left[\qdn\ba{#1}#3\\[1mm]#4\ea{\!;\:#2}\right]}

\newcommand{\binm}{\binom}

\newcommand{\nnm}{\nonumber}
\newcommand{\be}{\begin{equation}}
\newcommand{\ee}{\end{equation}}
\newcommand{\ba}{\begin{array}}
\newcommand{\ea}{\end{array}}
\newcommand{\bmn}{\begin{eqnarray}}
\newcommand{\emn}{\end{eqnarray}}
\newcommand{\bnm}{\begin{eqnarray*}}
\newcommand{\enm}{\end{eqnarray*}}
\newcommand{\bln}{\begin{subequations}}
\newcommand{\eln}{\end{subequations}}

\newcommand{\lam}{\lambda}

\newtheorem{thm}{Theorem}
\newtheorem{lemm}[thm]{Lemma}
\newtheorem{corl}[thm]{Corollary}

\newtheorem{entry}{Entry}

\newcommand{\bbtm}[4]{\bibitem{kn:#1}{#2,}~{#3,}~{#4.}}
\newcommand{\cito}[1]{\cite{kn:#1}}
\newcommand{\citu}[2]{\cite[#2]{kn:#1}}

\begin{document} 
{\fns
\title{Several transformation formulas involving bilateral basic hypergeometric series}
\author{Chuanan Wei, Tong Yu$^*$}


\dedicatory{
School of Biomedical Information and Engineering\\
  Hainan Medical University, Haikou 571199, China}
\thanks{\emph{Email addresses}:
      weichuanan78@163.com (C. Wei), ythainmc@163.com (T. Yu).\\
      The corresponding author$^*$.}

\address{ }

\begin{abstract}
In terms of the analytic continuation method, we prove three
transformation formulas involving bilateral basic hypergeometric
series. One of them is equivalent to Jouhet's result involving two
$_8\psi_8$ series and two $_8\phi_7$ series.
\end{abstract}

\maketitle\thispagestyle{empty}
\markboth{Chuanan Wei}
         {Bilateral basic hypergeometric series}

\section{Introduction}

\hspace{0.3cm} For any integer $n$ and complex numbers $x$, $q$ with
$|q|<1$, define the $q$-shifted factorial to be
 \bnm
(x;q)_{\infty}=\prod_{i=0}^{\infty}(1-xq^i),\quad
(x;q)_n=\frac{(x;q)_{\infty}}{(xq^n;q)_{\infty}}.
 \enm
For convenience, we shall also adopt the following notations:
 \bnm
&&(x_1,x_2,\ldots,x_r;q)_{\infty}=(x_1;q)_{\infty}(x_2;q)_{\infty}\cdots(x_r;q)_{\infty},\\
&&(x_1,x_2,\ldots,x_r;q)_{n}=(x_1;q)_{n}(x_2;q)_{n}\cdots(x_r;q)_{n}.
 \enm
Following Gasper and Rahman \cito{gasper}, define the bilateral
basic hypergeometric series by
 \bnm
&&\qdn\xqdn{_{r}\psi_s}\ffnk{cccccc}{q,z}{a_1,&a_2,&\ldots,&a_r}{b_1,&b_2,&\ldots,&b_s}
  =\sum_{k=-\infty}^{\infty}\frac{(a_1,a_2,\ldots,a_r;q)_k}{(b_1,b_2,\ldots,b_s;q)_k}\Big\{(-1)^kq^{\binm{k}{2}}\Big\}^{s-r}z^k.
 \enm
 When $b_1=q$, the bilateral basic hypergeometric
series becomes the unilateral basic hypergeometric series
 \bnm\quad
{_{r}\phi_{s-1}}\ffnk{cccccc}{q,z}{a_1,&a_2,&\ldots,&a_r}{&b_2,&\ldots,&b_s}
  =\sum_{k=0}^{\infty}\frac{(a_1,a_2,\ldots,a_r;q)_k}{(q,b_2,\ldots,b_s;q)_k}
\Big\{(-1)^kq^{\binm{k}{2}}\Big\}^{s-r}z^k.
 \enm
 Then Ramanujan's $_1\psi_1$ summation formula (cf. \citu{gasper}{Appendix II. 29}) and Bailey's
$_6\psi_6$ summation formula (cf. \citu{gasper}{Appendix II. 33})
can be stated as
 \bmn\label{ramanujan}
{_1\psi_1}\ffnk{cccccc}{q,z}{a}{b}=
\frac{(q,b/a,az,q/az;q)_{\infty}}{(b,q/a,z,b/az;q)_{\infty}},
 \emn
where $|b/a|<|z|<1$,
  \bmn\label{bailey}
  &&\xqdn{_6\psi_6}\ffnk{cccccccccc}{q,\frac{a^2q}{bcde}}
 {qa^{\frac{1}{2}},-qa^{\frac{1}{2}},b,c,d,e}
 {a^{\frac{1}{2}},-a^{\frac{1}{2}},aq/b,aq/c,aq/d,aq/e}
 \nnm\\&&\xqdn\:\,=\:
\frac{(q,aq,q/a,aq/bc,aq/bd,aq/be,aq/cd,aq/ce,aq/de;q)_{\infty}}{(q/b,q/c,q/d,q/e,aq/b,aq/c,aq/d,aq/e,a^2q/bcde;q)_{\infty}},
 \emn
provided $|a^2q/bcde|<1$.

\hspace{0.3cm} Via the analytic continuation method, Ismail
\cito{ismail} gave a simple proof of \eqref{ramanujan}. Askey and
Ismail \cito{askey} confirmed \eqref{bailey} in the same way. More
results from this method can be seen in Zhu \cito{zhu}.

\hspace{0.3cm} Inspired by the work just mentioned, we shall
establish the following three transformation formulas involving
bilateral basic hypergeometric series through the analytic
continuation method.

\begin{thm}\label{thm-aa}
Let $a,b,c,x$ be complex numbers such that $\max\{|aq/bcx|,
|aqx/bc|\}<1$. Then
  \bnm
  &&\xqdn{_2\psi_2}\ffnk{ccccccc}
 {q,\frac{aqx}{bc}}{b,c}{aq/b,aq/c}
  \\[1mm]
 &&\xqdn\:=\frac{(q/a,aq/bc,q/bx,q/cx,qx,aqx^2;q)_{\infty}}{(q/b,q/c,aq/bcx,q/ax,aqx,qx^2;q)_{\infty}}
  \\[1mm]
 &&\xqdn\:\times\: {_8\psi_8}\ffnk{cccccccc}
 {q,\frac{aqx}{bc}}{q(ax)^{\frac{1}{2}},-q(ax)^{\frac{1}{2}},(aq)^{\frac{1}{2}},-(aq)^{\frac{1}{2}},a^{\frac{1}{2}},-a^{\frac{1}{2}},bx,cx}
 {(ax)^{\frac{1}{2}},-(ax)^{\frac{1}{2}},x(aq)^{\frac{1}{2}},-x(aq)^{\frac{1}{2}},xqa^{\frac{1}{2}},-xqa^{\frac{1}{2}},aq/b,aq/c}.
 \enm
\end{thm}

\begin{thm}\label{thm-bb}
Let $a,b,c,x$ be complex numbers such that $\max\{|aq/bcx|,
|ax/bcq|\}<1$. Then
  \bnm
  &&\xqdn{_4\psi_4}\ffnk{ccccccc}
 {q,\frac{ax}{bcq}}{qa^{\frac{1}{2}},-qa^{\frac{1}{2}},b,c}{a^{\frac{1}{2}},-a^{\frac{1}{2}},aq/b,aq/c}
  \\[1mm]
 &&\xqdn\:=\frac{(q/a,aq/bc,q/bx,q/cx,x/q,ax^2;q)_{\infty}}{(q/b,q/c,aq/bcx,q/ax,aqx,x^2/q;q)_{\infty}}
  \\[1mm]
 &&\xqdn\:\times\: {_8\psi_8}\ffnk{cccccccc}
 {q,\frac{ax}{bcq}}{q(ax)^{\frac{1}{2}},-q(ax)^{\frac{1}{2}},(aq)^{\frac{1}{2}},-(aq)^{\frac{1}{2}},qa^{\frac{1}{2}},-qa^{\frac{1}{2}},bx,cx}
 {(ax)^{\frac{1}{2}},-(ax)^{\frac{1}{2}},x(aq)^{\frac{1}{2}},-x(aq)^{\frac{1}{2}},xa^{\frac{1}{2}},-xa^{\frac{1}{2}},aq/b,aq/c}.
 \enm
\end{thm}

\begin{thm}\label{thm-cc}
Let $a,b,c,d,e,f,g$ be complex numbers with $\max\{
|a^3q^2/bcdefg|,|aq/fg|\}<1$ and  ${\mu}=bcde/aq$. Then
  \bnm
  &&\xqdn{_8\psi_8}\ffnk{ccccccc}
 {q,\frac{a^3q^2}{bcdefg}}{qa^{\frac{1}{2}},-qa^{\frac{1}{2}},b,c,d,e,f,g}{a^{\frac{1}{2}},-a^{\frac{1}{2}},aq/b,aq/c,aq/d,aq/e,aq/f,aq/g}
  \\[1mm]
 &&\xqdn\:=\frac{(aq,q/a,aq/cd,aq/ce,aq/de,aq/fg,b/a,{\mu}q/c,{\mu}q/d,{\mu}q/e,aq/{\mu}f,aq/{\mu}g;q)_{\infty}}
 {(q/f,q/g,aq/c,aq/d,aq/e,bc/a,bd/a,be/a,b/{\mu},{\mu}q,q/{\mu},a^2q/{\mu}fg;q)_{\infty}}
  \\[1mm]
 &&\xqdn\:\times\: {_8\psi_8}\ffnk{cccccccc}
 {q,\frac{aq}{fg}}{q\mu^{\frac{1}{2}},-q\mu^{\frac{1}{2}},b,c,d,e,\mu f/a,{\mu}g/a}
 {\mu^{\frac{1}{2}},-\mu^{\frac{1}{2}},\mu q/b,{\mu}q/c,{\mu}q/d,{\mu}q/e,aq/f,aq/g}
\\[1mm]
  &&\xqdn\:+\:\frac{(q,aq,q/a,c,d,e,bq/c,bq/d,bq/e,bq/f,bq/g;q)_{\infty}}
  {(q/f,q/g,aq/b,aq/c,aq/d,aq/e,aq/f,aq/g,bc/a,bd/a,be/a;q)_{\infty}}
  \enm
  \bnm
  &&\xqdn\:\times\:\frac{(aq/bf,aq/bg,bcde/a^2q,a^2q^2/bcde;q)_{\infty}}{(q/b,b^2q/a,cde/aq,aq^2/cde;q)_{\infty}}\\[1mm]
 &&\xqdn\:\times\:{_8\phi_7}\ffnk{ccccccc}
 {q,\frac{a^3q^2}{bcdefg}}{b^2/a,qba^{-\frac{1}{2}},-qba^{-\frac{1}{2}},bc/a,bd/a,be/a,bf/a,bg/a}
  {ba^{-\frac{1}{2}},-ba^{-\frac{1}{2}},bq/c,bq/d,bq/e,bq/f,bq/g}.
 \enm
\end{thm}

\hspace{0.3cm} Fixing $x=1/c$ in Theorem \ref{thm-aa}, we obtain the
following result.

\begin{corl}\label{corl-aa}
Let $a,b,c$ be complex numbers such that $\max\{|aq/b|,
|aq/bc^2|\}<1$. Then
  \bnm
  &&\xqdn{_2\psi_2}\ffnk{ccccccc}
 {q,\frac{aq}{bc^2}}{b,c}{aq/b,aq/c}
  \\[1mm]
 &&\xqdn\:=\frac{(q,q/a,aq/bc,aq/c^2,cq/b;q)_{\infty}}{(q/b,q/c^2,aq/b,aq/c,cq/a;q)_{\infty}}
  \\[1mm]
 &&\xqdn\:\times\: {_8\phi_7}\ffnk{cccccccc}
 {q,\frac{aq}{bc^2}}{c/a,q(c/a)^{\frac{1}{2}},-q(c/a)^{\frac{1}{2}},c(q/a)^{\frac{1}{2}},-c(q/a)^{\frac{1}{2}},c/a^{\frac{1}{2}},-c/a^{\frac{1}{2}},b/a}
 {(c/a)^{\frac{1}{2}},-(c/a)^{\frac{1}{2}},(q/a)^{\frac{1}{2}},-(q/a)^{\frac{1}{2}},q/a^{\frac{1}{2}},-q/a^{\frac{1}{2}},cq/b}.
 \enm
\end{corl}

\hspace{0.3cm} Setting $x=1/c$ in Theorem \ref{thm-bb}, we get the
following conclusion.

\begin{corl}\label{corl-bb}
Let $a,b,c$ be complex numbers such that $\max\{|aq/b|,
|a/bc^2q|\}<1$. Then
  \bnm
  &&\xqdn{_4\psi_4}\ffnk{ccccccc}
 {q,\frac{a}{bc^2q}}{qa^{\frac{1}{2}},-qa^{\frac{1}{2}},b,c}{a^{\frac{1}{2}},-a^{\frac{1}{2}},aq/b,aq/c}
  \\[1mm]
 &&\xqdn\:=\frac{(q,q/a,1/cq,aq/bc,a/c^2,cq/b;q)_{\infty}}{(q/b,q/c,1/c^2q,aq/b,aq/c,cq/a;q)_{\infty}}
  \\[1mm]
 &&\xqdn\:\times\: {_8\phi_7}\ffnk{cccccccc}
 {q,\frac{a}{bc^2q}}{c/a,q(c/a)^{\frac{1}{2}},-q(c/a)^{\frac{1}{2}},c(q/a)^{\frac{1}{2}},-c(q/a)^{\frac{1}{2}},cq/a^{\frac{1}{2}},-cq/a^{\frac{1}{2}},b/a}
 {(c/a)^{\frac{1}{2}},-(c/a)^{\frac{1}{2}},(q/a)^{\frac{1}{2}},-(q/a)^{\frac{1}{2}},1/a^{\frac{1}{2}},-1/a^{\frac{1}{2}},cq/b}.
 \enm
\end{corl}

\hspace{0.3cm} When $aq=de$, Theorem \ref{thm-cc} becomes
  \bnm
  &&\xqdn{_6\psi_6}\ffnk{cccccccccc}{q,\frac{a^2q}{bcfg}}
 {qa^{\frac{1}{2}},-qa^{\frac{1}{2}},b,c,f,g}
 {a^{\frac{1}{2}},-a^{\frac{1}{2}},aq/b,aq/c,aq/f,aq/g}
 \\[1mm]&&\xqdn\:\,=\:
\frac{(q,aq,q/a,aq/bc,aq/bf,aq/bg,bq/c,bq/f,bq/g;q)_{\infty}}{(q/b,q/c,q/f,q/g,aq/b,aq/c,aq/f,aq/g,b^2q/a;q)_{\infty}}
 \\[1mm]&&\xqdn\:\,\times\:\,
 {_6\phi_5}\ffnk{cccccccccc}{q,\frac{a^2q}{bcfg}}
 {b^2/a,qba^{-\frac{1}{2}},-qba^{-\frac{1}{2}},bc/a,bf/a,bg/a}
 {ba^{-\frac{1}{2}},-ba^{-\frac{1}{2}},bq/c,bq/f,bq/g}.
 \enm
Calculating the series on the right-hand side by Rogers' $_6\phi_5$
summation formula (cf. \citu{gasper}{Appendix II. 21}):
 \bnm
 { _6\phi_5}\ffnk{cccccccccc}{q,\frac{aq}{bcd}}
 {a,qa^{\frac{1}{2}},-qa^{\frac{1}{2}},b,c,d}
 {a^{\frac{1}{2}},-a^{\frac{1}{2}},aq/b,aq/c,aq/d}
 =
\frac{(aq,aq/bc,aq/bd,aq/cd;q)_{\infty}}{(aq/b,aq/c,aq/d,aq/bcd;q)_{\infty}},
 \enm
we arrive at Bailey's $_6\psi_6$ summation formula \eqref{bailey}.

\hspace{0.3cm} By means of Cauchy's method, Jouhet \cito{jouhet}
found the identity involving two $_8\psi_8$ series and two
$_8\phi_7$ series:
 \bmn\label{jouhet}
  &&\xqdn{_8\psi_8}\ffnk{ccccccc}
 {q,c}{qa^{\frac{1}{2}},-qa^{\frac{1}{2}},b,d,e,f,g,h}{a^{\frac{1}{2}},-a^{\frac{1}{2}},aq/b,aq/d,aq/e,aq/f,aq/g,aq/h}
  \nnm\\[1mm]
 &&\xqdn\:=\frac{(aq,q/a,\lambda c/a,aq/\lambda d,aq/\lambda e,b/a,bf/\lambda,bg/\lambda,bh/\lambda,\lambda q/f,\lambda q/g,\lambda q/h;q)_{\infty}}
 {(\lambda q,q/\lambda, c,q/d,q/e,b/\lambda,bf/a,bg/a,bh/a,aq/f,aq/g,aq/h;q)_{\infty}}
 \nnm\\[1mm]
  &&\xqdn\:\times\: {_8\psi_8}\ffnk{cccccccc}
 {q,\frac{\lambda c}{a}}{q\lambda^{\frac{1}{2}},-q\lambda^{\frac{1}{2}},b,\lambda d/a,\lambda e/a,f,g,h}
 {\lambda^{\frac{1}{2}},-\lambda^{\frac{1}{2}},\lambda q/b,aq/d,aq/e,\lambda q/f,\lambda q/g,\lambda
 q/h}
 \nnm\\[1mm]
  &&\xqdn\:+\:\frac{b}{a}\frac{(q, q/a, c/b, aq, bq/c, bq/d, bq/e, bq/f, bq/g, bq/h, d, e;q)_{\infty}}
  {(q/b, c/a, b^2q/a, aq/b, aq/c, bd/a, be/a, bf/a, bg/a, bh/a, aq/d, aq/e;q)_{\infty}}
  \nnm\\[1mm]
  &&\xqdn\:\times\:\frac{(f,g,h;q)_{\infty}}{(aq/f, aq/g, aq/h;q)_{\infty}}
 \nnm\\[1mm]
 &&\nnm\xqdn\:\times\:{_8\phi_7}\ffnk{ccccccc}
 {q,c}{b^2/a,qba^{-\frac{1}{2}},-qba^{-\frac{1}{2}},bd/a,be/a,bf/a,bg/a,bh/a}
  {ba^{-\frac{1}{2}},-ba^{-\frac{1}{2}},bq/d,bq/e,bq/f,bq/g,bq/h}\\[1mm]
&&\nnm\xqdn\:+\:
 \frac{(q, q/a, b/a, aq/\lambda d, aq/\lambda e, \lambda c/ab, aq, f, g, h,\lambda c/a, \lambda d/a, \lambda e/a;q)_{\infty}}
  {(c, c/a, q/b, q/d, q/e, bq/\lambda, b^2q/\lambda, aq/f, aq/g, aq/h, bd/a, be/a, bf/a;q)_{\infty}}
  \nnm\\[1mm]
  &&\xqdn\:\times\:\frac{(bq/f, bq/g, bq/h, abq/\lambda c, abq/\lambda d, abq/\lambda e;q)_{\infty}}
  {( bg/a, bh/a, \lambda/b, aq/c, aq/d, aq/e;q)_{\infty}}
 \nnm\\[1mm]
 &&\xqdn\:\times\:{_8\phi_7}\ffnk{ccccccc}
 {q,\frac{\lambda c}{a}}{b^2/\lambda,qb\lambda^{-\frac{1}{2}},-qb\lambda^{-\frac{1}{2}},bd/a,be/a,bf/\lambda,bg/\lambda,bh/\lambda}
  {b\lambda^{-\frac{1}{2}},-b\lambda^{-\frac{1}{2}},abq/\lambda d,abq/\lambda e,bq/f,bq/g,bq/h},
 \emn
where $c=a^3q^2/bdefgh$, $\lambda=a^2q/cde$ and $\max\{|c|, |\lambda
c/a|\}<1$.

\hspace{0.3cm} Recall the transformation formula (cf.
\citu{gasper}{Appendix III. 23}):
 \bmn\label{two-term}
  &&\xxqdn\xxqdn{_8\phi_7}\ffnk{ccccccc}
 {q,\frac{a^2q^2}{bcdef}}{a,qa^{\frac{1}{2}},-qa^{\frac{1}{2}},b,c,d,e,f}{a^{\frac{1}{2}},-a^{\frac{1}{2}},aq/b,aq/c,aq/d,aq/e,aq/f}
 \nnm\\[1mm]\nnm
 &&\xxqdn\xxqdn\:=
\frac{(aq,aq/ef,\lam q/e,\lam q/f;q)_{\infty}}
 {(aq/e,aq/f,\lam q,\lam q/ef;q)_{\infty}}
  \\[1mm]
 &&\xxqdn\xxqdn\:\times\:{_8\phi_7}\ffnk{ccccccc}
 {q,\frac{aq}{ef}}{\lam,q\lam^{\frac{1}{2}},-q\lam^{\frac{1}{2}},\lam b/a,\lam c/a,\lam d/a,e,f}
  {\lam^{\frac{1}{2}},-\lam^{\frac{1}{2}},aq/b,aq/c,aq/d,\lam q/e,\lam q/f}
 \emn
with $\lam=a^2q/bcd$ and the three-term relation (cf.
\citu{gasper}{Exercise 2.16}):
 \bnm
S(x\lambda,x/\lambda,\mu\nu,\mu/v)-S(x\nu,x/\nu,\lambda\mu,\mu/\lambda)=\frac{\mu}{\lambda}S(x\mu,x/\mu,\lambda\nu,\lambda/v),
 \enm
where
 \bnm
S(\lambda,\mu,\nu,\rho)=(\lambda,q/\lambda,\mu,q/\mu,\nu,q/\nu,\rho,q/\rho;q)_{\infty}.
\enm
 On the basis of the last two equations, it is not
 difficult to understand the equivalence of Theorem \ref{thm-cc} and
 Fouhet's result \eqref{jouhet}.

\hspace{0.3cm} The rest of the paper is arranged as follows. The
proofs of Theorems \ref{thm-aa} and \ref{thm-bb}  will respectively
be displayed in Sections 2 and 3. In section 4, we shall give a new
proof of Theorem \ref{thm-cc}.

\section{Proof of Theorem \ref{thm-aa}}
\hspace{0.3cm}In order to prove Theorem \ref{thm-aa}, we need the
following two lemmas.

\begin{lemm}\label{lemma-add}  For all integers $s,t\in\mathbb{Z}$, it holds
\bnm
 \frac{(a;q)_{s+t}}{(b;q)_{s+t}}=\frac{(a;q)_{s}}{(b;q)_{s}}\frac{(aq^s;q)_{t}}{(bq^s;q)_{t}}.
 \enm
\end{lemm}

\begin{lemm}[cf. \citu{lang}{p. 90}; see also \citu{zhu}{p.
741}] \label{lemma} Let $U$ be a connected open set and $f$, $g$ be
analytic on $U$. If $f$ and $g$ agree infinitely often near an
interior point of $U$, then we have $f(z)=g(z)$ for all $z\in U$.
\end{lemm}

\begin{proof}[Proof of Theorem \ref{thm-aa}]
A known transformation formula for unilateral basic hypergeometric
series (cf. \citu{gasper}{Equation (3.4.7)}) reads
 \bnm
  &&\xqdn{_2\phi_1}\ffnk{ccccccc}{q,\frac{qx}{b^2}}{a,b}{aq/b}
  =\frac{(xq/b,aqx^2/b^2;q)_{\infty}}{(aqx/b,qx^2/b^2;q)_{\infty}}
  \\[1mm]
 &&\xqdn\:\times\: {_8\phi_7}\ffnk{cccccccc}
 {q,\frac{qx}{b^2}}{ax/b,q(ax/b)^{\frac{1}{2}},-q(ax/b)^{\frac{1}{2}},(aq)^{\frac{1}{2}},-(aq)^{\frac{1}{2}},a^{\frac{1}{2}},-a^{\frac{1}{2}},x}
 {(ax/b)^{\frac{1}{2}},-(ax/b)^{\frac{1}{2}},x(aq)^{\frac{1}{2}}/b,-x(aq)^{\frac{1}{2}}/b,xqa^{\frac{1}{2}}/b,-xqa^{\frac{1}{2}}/b,aq/b}.
 \enm
Performing the replacements $a\to aq^{-2m}, b\to bq^{-m}, x \to
bxq^{-m}$ in this equation, we obtain
 \bmn\label{equation-a}
  &&\xxqdn{_2\phi_1}\ffnk{ccccccc}{q,\frac{q^{1+m}x}{b}}{q^{-2m}a,q^{-m}b}{q^{1-m}a/b}
  =\frac{(qx,q^{1-2m}ax^2;q)_{\infty}}{(qx^2,q^{1-2m}ax;q)_{\infty}}
  \nnm\\[1mm]
 &&\xxqdn\:\times\: {_8\phi_7}\ffnk{cccccccc}
 {q,\frac{q^{1+m}x}{b}}{q^{-2m}ax,q^{1-m}(ax)^{\frac{1}{2}},-q^{1-m}(ax)^{\frac{1}{2}},
 q^{-m}(aq)^{\frac{1}{2}},-q^{-m}(aq)^{\frac{1}{2}},\\[1mm]
 q^{-m}a^{\frac{1}{2}},-q^{-m}a^{\frac{1}{2}},q^{-m}bx}
 {q^{-m}(ax)^{\frac{1}{2}},-q^{-m}(ax)^{\frac{1}{2}},q^{-m}x(aq)^{\frac{1}{2}},-q^{-m}x(aq)^{\frac{1}{2}},\\[1mm]
 q^{1-m}xa^{\frac{1}{2}},-q^{1-m}xa^{\frac{1}{2}},q^{1-m}a/b}.
 \emn

\hspace{0.3cm}Let $m$ be a nonnegative integer throughout the paper.
On one hand, it is easy to deduce, from Lemma \ref{lemma-add}, the
following relation:
  \bmn\label{equation-b}
  &&\xxqdn{_2\phi_1}\ffnk{ccccccc}{q,\frac{q^{1+m}x}{b}}{q^{-2m}a,q^{-m}b}{q^{1-m}a/b}
  \nnm\\[1mm]
 &&\xxqdn\:=\sum_{k=-m}^{\infty}\frac{(q^{-2m}a,q^{-m}b;q)_{k+m}}{(q,q^{1-m}a/b;q)_{k+m}}\bigg(\frac{q^{1+m}x}{b}\bigg)^{k+m}
\nnm\\[1mm]
 &&\xxqdn\:=\frac{(q^{-2m}a,q^{-m}b;q)_{m}}{(q,q^{1-m}a/b;q)_{m}}\bigg(\frac{q^{1+m}x}{b}\bigg)^{m}
{_2\psi_2}\ffnk{ccccccc}{q,\frac{q^{1+m}x}{b}}{q^{-m}a,b}{q^{1+m},aq/b}.
 \emn
On the other hand, we derive, from Lemma \ref{lemma-add}, the
following relation:
 \bmn
 &&\xqdn{_8\phi_7}\ffnk{cccccccc}
 {q,\frac{q^{1+m}x}{b}}{q^{-2m}ax,q^{1-m}(ax)^{\frac{1}{2}},-q^{1-m}(ax)^{\frac{1}{2}},
 q^{-m}(aq)^{\frac{1}{2}},-q^{-m}(aq)^{\frac{1}{2}},\\[1mm]
 q^{-m}a^{\frac{1}{2}},-q^{-m}a^{\frac{1}{2}},q^{-m}bx}
 {q^{-m}(ax)^{\frac{1}{2}},-q^{-m}(ax)^{\frac{1}{2}},q^{-m}x(aq)^{\frac{1}{2}},-q^{-m}x(aq)^{\frac{1}{2}},\\[1mm]
 q^{1-m}xa^{\frac{1}{2}},-q^{1-m}xa^{\frac{1}{2}},q^{1-m}a/b}
 \nnm
 \emn
\bmn
 &&\xqdn\:
=\sum_{n=-m}^{\infty}\frac{(q^{-2m}ax,q^{1-m}(ax)^{\frac{1}{2}},-q^{1-m}(ax)^{\frac{1}{2}},
 q^{-m}(aq)^{\frac{1}{2}},-q^{-m}(aq)^{\frac{1}{2}};q)_{n+m}}
 {(q,q^{-m}(ax)^{\frac{1}{2}},-q^{-m}(ax)^{\frac{1}{2}},q^{-m}x(aq)^{\frac{1}{2}},-q^{-m}x(aq)^{\frac{1}{2}};q)_{n+m}}
 \nnm\\[1mm]\nnm
 &&\xqdn\qquad\times\:
 \frac{(q^{-m}a^{\frac{1}{2}},-q^{-m}a^{\frac{1}{2}},q^{-m}bx;q)_{n+m}}{(q^{1-m}xa^{\frac{1}{2}},-q^{1-m}xa^{\frac{1}{2}},q^{1-m}a/b;q)_{n+m}}
 \bigg(\frac{q^{1+m}x}{b}\bigg)^{n+m}
\\[1mm]
 &&\xqdn\:
=\frac{(q^{-2m}ax,q^{1-m}(ax)^{\frac{1}{2}},-q^{1-m}(ax)^{\frac{1}{2}},
 q^{-m}(aq)^{\frac{1}{2}},-q^{-m}(aq)^{\frac{1}{2}};q)_{m}}
 {(q,q^{-m}(ax)^{\frac{1}{2}},-q^{-m}(ax)^{\frac{1}{2}},q^{-m}x(aq)^{\frac{1}{2}},-q^{-m}x(aq)^{\frac{1}{2}};q)_{m}}
 \nnm\\[1mm]
 &&\xqdn\:\times\:
 \frac{(q^{-m}a^{\frac{1}{2}},-q^{-m}a^{\frac{1}{2}},q^{-m}bx;q)_{m}}{(q^{1-m}xa^{\frac{1}{2}},-q^{1-m}xa^{\frac{1}{2}},q^{1-m}a/b;q)_{m}}
 \bigg(\frac{q^{1+m}x}{b}\bigg)^{m}
  \nnm\\[1mm]
 &&\xqdn\:\times\: {_8\psi_8}\ffnk{cccccccc}
 {q,\frac{q^{1+m}x}{b}}{q^{-m}ax,q(ax)^{\frac{1}{2}},-q(ax)^{\frac{1}{2}},(aq)^{\frac{1}{2}},-(aq)^{\frac{1}{2}},a^{\frac{1}{2}},-a^{\frac{1}{2}},bx}
 {q^{1+m},(ax)^{\frac{1}{2}},-(ax)^{\frac{1}{2}},x(aq)^{\frac{1}{2}},-x(aq)^{\frac{1}{2}},xqa^{\frac{1}{2}},-xqa^{\frac{1}{2}},aq/b}.
\label{equation-c}
 \emn
Substituting \eqref{equation-b} and \eqref{equation-c} into
\eqref{equation-a} leads us to
 \bmn\label{equation-d}
  &&\xxqdn{_2\psi_2}\ffnk{ccccccc}{q,\frac{q^{1+m}x}{b}}{q^{-m}a,b}{q^{1+m},aq/b}
  \nnm\\[1mm]
 &&\xxqdn\:=\frac{(q^{1+m}/b,q^{1+m}/ax,q/a,q/bx,qx,aqx^2,;q)_{\infty}}{(q^{1+m}/a,q^{1+m}/bx,q/b,q/ax,qx^2,aqx;q)_{\infty}}
  \nnm\\[1mm]
 &&\xxqdn\:\times\: {_8\psi_8}\ffnk{cccccccc}
 {q,\frac{q^{1+m}x}{b}}{q^{-m}ax,q(ax)^{\frac{1}{2}},-q(ax)^{\frac{1}{2}},(aq)^{\frac{1}{2}},-(aq)^{\frac{1}{2}},a^{\frac{1}{2}},-a^{\frac{1}{2}},bx}
 {q^{1+m},(ax)^{\frac{1}{2}},-(ax)^{\frac{1}{2}},x(aq)^{\frac{1}{2}},-x(aq)^{\frac{1}{2}},xqa^{\frac{1}{2}},-xqa^{\frac{1}{2}},aq/b}.
 \emn

\hspace{0.3cm}Define two functions $f(z)$ and $g(z)$ to be
 \bnm
&&\xqdn\xxqdn\qqdn f(z)={_2\psi_2}\ffnk{ccccccc}{q,\frac{xz}{b}}{aq/z,b}{z,aq/b}\\[1mm]
 &&\xxqdn=\sum_{k=0}^{\infty}\frac{(b;q)_{k}\prod_{i=1}^k(z-aq^i)}
 {(aq/b,z;q)_{k}}\bigg(\frac{x}{b}\bigg)^{k}
+\sum_{k=1}^{\infty}\frac{(b/a;q)_{k}\prod_{i=1}^k(z-q^i)}
 {(q/b,z/a;q)_{k}}\bigg(\frac{1}{bx}\bigg)^{k},
 \enm
\bnm
 &&\xqdn g(z)=\frac{(z/b,z/ax,q/a,q/bx,qx,aqx^2,;q)_{\infty}}{(z/a,z/bx,q/b,q/ax,qx^2,aqx;q)_{\infty}}
  \nnm\\[1mm]
 &&\:\:\:\times\: {_8\psi_8}\ffnk{cccccccc}
 {q,\frac{xz}{b}}{aqx/z,q(ax)^{\frac{1}{2}},-q(ax)^{\frac{1}{2}},(aq)^{\frac{1}{2}},-(aq)^{\frac{1}{2}},a^{\frac{1}{2}},-a^{\frac{1}{2}},bx}
 {z,(ax)^{\frac{1}{2}},-(ax)^{\frac{1}{2}},x(aq)^{\frac{1}{2}},-x(aq)^{\frac{1}{2}},xqa^{\frac{1}{2}},-xqa^{\frac{1}{2}},aq/b}\\[1mm]
 &&\:\:\,=\frac{(z/b,z/ax,q/a,q/bx,qx,aqx^2,;q)_{\infty}}{(z/a,z/bx,q/b,q/ax,qx^2,aqx;q)_{\infty}}\\[1mm]
  &&\:\:\:\times\:\sum_{n=0}^{\infty}\frac{(q(ax)^{\frac{1}{2}},-q(ax)^{\frac{1}{2}},(aq)^{\frac{1}{2}},-(aq)^{\frac{1}{2}},a^{\frac{1}{2}},-a^{\frac{1}{2}},bx;q)_{n}\prod_{i=1}^n(z-q^iax)}
 {((ax)^{\frac{1}{2}},-(ax)^{\frac{1}{2}},x(aq)^{\frac{1}{2}},-x(aq)^{\frac{1}{2}},xqa^{\frac{1}{2}},-xqa^{\frac{1}{2}},aq/b,z;q)_{n}}\bigg(\frac{x}{b}\bigg)^{n}\\[1mm]
&&\:\:\:+\frac{(z/b,z/ax,q/a,q/bx,qx,aqx^2,;q)_{\infty}}{(z/a,z/bx,q/b,q/ax,qx^2,aqx;q)_{\infty}}\\[1mm]
 &&\:\:\:\times\:\sum_{n=1}^{\infty}\frac{(q/(ax)^{\frac{1}{2}},-q/(ax)^{\frac{1}{2}},(q/a)^{\frac{1}{2}}/x,-(q/a)^{\frac{1}{2}}/x,1/xa^{\frac{1}{2}},-1/xa^{\frac{1}{2}},b/a;q)_{n}}
 {(1/(ax)^{\frac{1}{2}},-1/(ax)^{\frac{1}{2}},(q/a)^{\frac{1}{2}},-(q/a)^{\frac{1}{2}},q/a^{\frac{1}{2}},-q/a^{\frac{1}{2}},q/bx;q)_{n}}\\[1mm]
&&\:\:\:\times\:\frac{\prod_{i=1}^n(z-q^i)}
 {(z/ax;q)_{n}}\bigg(\frac{x}{b}\bigg)^{n}.
\enm
  Then \eqref{equation-d} shows that
 \bmn \label{equation-e}
 f(z)=g(z)
 \emn
for $z=q^{1+m}$. Based on Lemma \ref{lemma}, \eqref{equation-e}
holds for all $|z|<\min\{1,|a|,|ax|,|bx|\}$. By the analytic
continuation method, the restriction on $z$ can by relaxed. Thus we
get
  \bnm
&&{_2\psi_2}\ffnk{ccccccc}{q,\frac{xz}{b}}{aq/z,b}{z,aq/b}\\[1mm]
 &&\:=\frac{(z/b,z/ax,q/a,q/bx,qx,aqx^2,;q)_{\infty}}{(z/a,z/bx,q/b,q/ax,qx^2,aqx;q)_{\infty}}
  \nnm\\[1mm]
 &&\:\times\: {_8\psi_8}\ffnk{cccccccc}
 {q,\frac{xz}{b}}{aqx/z,q(ax)^{\frac{1}{2}},-q(ax)^{\frac{1}{2}},(aq)^{\frac{1}{2}},-(aq)^{\frac{1}{2}},a^{\frac{1}{2}},-a^{\frac{1}{2}},bx}
 {z,(ax)^{\frac{1}{2}},-(ax)^{\frac{1}{2}},x(aq)^{\frac{1}{2}},-x(aq)^{\frac{1}{2}},xqa^{\frac{1}{2}},-xqa^{\frac{1}{2}},aq/b}.
\enm Replying $z$ by $aq/c$ in it, we arrive at Theorem
\ref{thm-aa}.
\end{proof}

\section{Proof of Theorem \ref{thm-bb}}
\hspace{0.3cm} Now we begin to prove Theorem \ref{thm-bb} with the
help of Lemma \ref{lemma} and the following transformation formula
for unilateral basic hypergeometric series (cf.
\citu{gasper}{Equation (3.4.8)}):
 \bnm
  &&\xqdn{_4\phi_3}\ffnk{ccccccc}
 {q,\frac{x}{b^2q}}{a,qa^{\frac{1}{2}},-qa^{\frac{1}{2}},b}{a^{\frac{1}{2}},-a^{\frac{1}{2}},aq/b}
  =\frac{(ax^2/b^2,x/bq;q)_{\infty}}{(aqx/b,x^2/b^2q;q)_{\infty}}
  \\[1mm]
 &&\xqdn\:\times\: {_8\phi_7}\ffnk{cccccccc}
 {q,\frac{x}{b^2q}}{ax/b,q(ax/b)^{\frac{1}{2}},-q(ax/b)^{\frac{1}{2}},(aq)^{\frac{1}{2}},-(aq)^{\frac{1}{2}},qa^{\frac{1}{2}},-qa^{\frac{1}{2}},x}
 {(ax/b)^{\frac{1}{2}},-(ax/b)^{\frac{1}{2}},x(aq)^{\frac{1}{2}}/b,-x(aq)^{\frac{1}{2}}/b,xa^{\frac{1}{2}}/b,-xa^{\frac{1}{2}}/b,aq/b}.
 \enm

\begin{proof}[Proof of Theorem \ref{thm-bb}]
Employing the replacements $a\to aq^{-2m}, b\to bq^{-m}, x \to
bxq^{-m}$ in the last equation, we obtain
 \bnm
  &&\xxqdn{_4\phi_3}\ffnk{ccccccc}{q,\frac{q^{m-1}x}{b}}{q^{-2m}a,q^{1-m}a^{\frac{1}{2}},-q^{1-m}a^{\frac{1}{2}},q^{-m}b}
  {q^{-m}a^{\frac{1}{2}},-q^{-m}a^{\frac{1}{2}},q^{1-m}a/b}
  =\frac{(x/q,q^{-2m}ax^2;q)_{\infty}}{(x^2/q,q^{1-2m}ax;q)_{\infty}}
  \nnm\\[1mm]
 &&\xxqdn\:\times\: {_8\phi_7}\ffnk{cccccccc}
 {q,\frac{q^{m-1}x}{b}}{q^{-2m}ax,q^{1-m}(ax)^{\frac{1}{2}},-q^{1-m}(ax)^{\frac{1}{2}},
 q^{-m}(aq)^{\frac{1}{2}},-q^{-m}(aq)^{\frac{1}{2}},\\[1mm]
 q^{1-m}a^{\frac{1}{2}},-q^{1-m}a^{\frac{1}{2}},q^{-m}bx}
 {q^{-m}(ax)^{\frac{1}{2}},-q^{-m}(ax)^{\frac{1}{2}},q^{-m}x(aq)^{\frac{1}{2}},-q^{-m}x(aq)^{\frac{1}{2}},\\[1mm]
 q^{-m}xa^{\frac{1}{2}},-q^{-m}xa^{\frac{1}{2}},q^{1-m}a/b}.
 \enm
 According to Lemma \ref{lemma-add}, we can reformulate it
 as
 \bnm
  &&\xqdn{_4\psi_4}\ffnk{ccccccc}{q,\frac{q^{m-1}x}{b}}{q^{-m}a,qa^{\frac{1}{2}},-qa^{\frac{1}{2}},b}{q^{1+m},a^{\frac{1}{2}},-a^{\frac{1}{2}},aq/b}
  \nnm\\[1mm]
 &&\xqdn\:=\frac{(q^{1+m}/b,q^{1+m}/ax,q/a,q/bx,x/q,ax^2,;q)_{\infty}}{(q^{1+m}/a,q^{1+m}/bx,q/b,q/ax,x^2/q,aqx;q)_{\infty}}
  \nnm\\[1mm]
 &&\xqdn\:\times\:  {_8\psi_8}\ffnk{cccccccc}
 {q,\frac{q^{m-1}x}{b}}{q^{-m}ax,q(ax)^{\frac{1}{2}},-q(ax)^{\frac{1}{2}},(aq)^{\frac{1}{2}},-(aq)^{\frac{1}{2}},qa^{\frac{1}{2}},-qa^{\frac{1}{2}},bx}
 {q^{1+m},(ax)^{\frac{1}{2}},-(ax)^{\frac{1}{2}},x(aq)^{\frac{1}{2}},-x(aq)^{\frac{1}{2}},xa^{\frac{1}{2}},-xa^{\frac{1}{2}},aq/b}.
 \enm
This transformation tells us that Theorem \ref{thm-bb} is true for
$aq/c=q^{1+m}$. Therefore, we can prove Theorem \ref{thm-bb} via
Lemma \ref{lemma} and the analytic continuation argument.
\end{proof}

\section{A new Proof of Theorem \ref{thm-cc}}
\hspace{0.3cm} For the sake of proving Theorem \ref{thm-cc}, we
require the following lemma.

\begin{lemm}\label{lemm-b}
Let $a,b,c,d,e,f$ be complex numbers. Then
 \bnm
  &&\qqdn\xxqdn\xqdn{_8\phi_7}\ffnk{ccccccc}
 {q,\frac{a^2q^2}{bcdef}}{a,qa^{\frac{1}{2}},-qa^{\frac{1}{2}},b,c,d,e,f}{a^{\frac{1}{2}},-a^{\frac{1}{2}},aq/b,aq/c,aq/d,aq/e,aq/f}
  \\[1mm]
 &&\qqdn\xxqdn\xqdn\:=\frac{(aq,aq/cd,aq/ce,aq/de,b/a,{\mu}q/c,{\mu}q/d,{\mu}q/e;q)_{\infty}}{(aq/c,aq/d,aq/e,bc/a,bd/a,be/a,b/{\mu},{\mu}q;q)_{\infty}}\\
 &&\qqdn\xxqdn\xqdn\:\times\:{_8\phi_7}\ffnk{ccccccc}
 {q,\frac{q}{f}}{{\mu},q{\mu}^{\frac{1}{2}},-q{\mu}^{\frac{1}{2}},{\mu}f/a,b,c,d,e}
  {{\mu}^{\frac{1}{2}},-{\mu}^{\frac{1}{2}},aq/f,{\mu}q/b,{\mu}q/c,{\mu}q/d,{\mu}q/e}\\[1mm]
  &&\qqdn\xxqdn\xqdn\:+\:\frac{(aq,bq/a,bq/c,bq/d,bq/e,bq/f,c,d,e;q)_{\infty}}
  {(aq/b,aq/c,aq/d,aq/e,aq/f,bc/a,bd/a,be/a,cde/aq;q)_{\infty}}\\[1mm]
  &&\qqdn\xxqdn\xqdn\:\times\:\frac{(aq/bf,bcde/a^2q,a^2q^2/bcde;q)_{\infty}}{(aq^2/cde,q/f,b^2q/a;q)_{\infty}}\\[1mm]
 &&\qqdn\xxqdn\xqdn\:\times\:{_8\phi_7}\ffnk{ccccccc}
 {q,\frac{a^2q^2}{bcdef}}{b^2/a,qba^{-\frac{1}{2}},-qba^{-\frac{1}{2}},b,bc/a,bd/a,be/a,bf/a}
  {ba^{-\frac{1}{2}},-ba^{-\frac{1}{2}},bq/a,bq/c,bq/d,bq/e,bq/f},
 \enm
where ${\mu}=bcde/aq$ and $\max\{|q/f|, |a^2q^2/bcdef|\}<1$.
\end{lemm}

\begin{proof}
It is clear from \eqref{two-term} that
 \bnm
  &&\xqdn{_8\phi_7}\ffnk{ccccccc}
 {q,\frac{bd}{a}}{ef/c,q(ef/c)^{\frac{1}{2}},-q(ef/c)^{\frac{1}{2}},aq/bc,aq/cd,ef/a,e,f}
  {(ef/c)^{\frac{1}{2}},-(ef/c)^{\frac{1}{2}},bef/a,def/a,aq/c,fq/c,eq/c}
 \\[1mm]&&\xqdn\:=
\frac{(q/c,efq/c,bde/a,bdf/a;q)_{\infty}}{(eq/c,fq/c,bd/a,bdef/a;q)_{\infty}}
  \\[1mm]
 &&\xqdn\:\times\:{_8\phi_7}\ffnk{ccccccc}
 {q,\frac{q}{c}}{bdef/aq,q(bdef/aq)^{\frac{1}{2}},-q(bdef/aq)^{\frac{1}{2}},bcdef/a^2q,b,d,e,f}
  {(bdef/aq)^{\frac{1}{2}},-(bdef/aq)^{\frac{1}{2}},aq/c,def/a,bef/a,bdf/a,bde/a}.
 \enm
Putting it into the transformation formula (cf.
\citu{gasper}{Appendix III. 37}):
 \bnm
  &&\xxqdn{_8\phi_7}\ffnk{ccccccc}
 {q,\frac{a^2q^2}{bcdef}}{a,qa^{\frac{1}{2}},-qa^{\frac{1}{2}},b,c,d,e,f}{a^{\frac{1}{2}},-a^{\frac{1}{2}},aq/b,aq/c,aq/d,aq/e,aq/f}\\[1mm]
 &&\xxqdn\:=\frac{(aq,aq/de,aq/df,aq/ef,eq/c,fq/c,b/a,bef/a;q)_{\infty}}{(aq/d,aq/e,aq/f,aq/def,q/c,efq/c,be/a,bf/a;q)_{\infty}}\\[1mm]
 &&\xxqdn\:\times\:{_8\phi_7}\ffnk{ccccccc}
 {q,\frac{bd}{a}}{ef/c,q(ef/c)^{\frac{1}{2}},-q(ef/c)^{\frac{1}{2}},aq/bc,aq/cd,ef/a,e,f}
  {(ef/c)^{\frac{1}{2}},-(ef/c)^{\frac{1}{2}},bef/a,def/a,aq/c,fq/c,eq/c} \\[1mm]
  &&\xxqdn\:+\:\frac{b}{a}\frac{(aq,bq/a,bq/c,bq/d,bq/e,bq/f,d,e,f;q)_{\infty}}
  {(aq/b,aq/c,aq/d,aq/e,aq/f,bd/a,be/a,bf/a,def/a;q)_{\infty}}\\[1mm]
  &&\xxqdn\:\times\:\frac{(aq/bc,bdef/a^2,a^2q/bdef;q)_{\infty}}{(aq/def,q/c,b^2q/a;q)_{\infty}}\\[1mm]
 &&\xxqdn\:\times\:{_8\phi_7}\ffnk{ccccccc}
 {q,\frac{a^2q^2}{bcdef}}{b^2/a,qba^{-\frac{1}{2}},-qba^{-\frac{1}{2}},b,bc/a,bd/a,be/a,bf/a}
  {ba^{-\frac{1}{2}},-ba^{-\frac{1}{2}},bq/a,bq/c,bq/d,bq/e,bq/f}
 \enm
 and then interchanging the parameters $c$ and $f$, we
decuce Lemma \ref{lemm-b}.
\end{proof}

\begin{proof}[Proof of Theorem \ref{thm-cc}]
Performing the replacements $a\to aq^{-2m},b\to bq^{-m}, c\to
cq^{-m}, d\to dq^{-m}, e\to eq^{-m}$, $f\to fq^{-m}$ in Lemma
\ref{lemm-b}, we derive
 \bnm
&&\xqdn{_8\phi_7}\!\ffnk{ccccccc}
 {q,\frac{q^{2+m}a^2}{bcdef}}{q^{-2m}a,q^{1-m}a^{\frac{1}{2}},-q^{1-m}a^{\frac{1}{2}},q^{-m}b,q^{-m}c,q^{-m}d,q^{-m}e,q^{-m}f}
 {q^{-m}\!a^{\frac{1}{2}},-q^{-m}\!a^{\frac{1}{2}},q^{1-m}a/b,q^{1-m}a/c,q^{1-m}a/d,q^{1-m}a/e,q^{1-m}a/f}
 \nnm\\[1mm]\nnm
 &&\xqdn\:=\:\frac{(q^{1-2m}a,aq/cd,aq/ce,aq/de,q^mb/a,q^{1-m}{\mu}/c,q^{1-m}{\mu}/d,q^{1-m}{\mu}/e;q)_{\infty}}
{(q^{1-m}a/c,q^{1-m}a/d,q^{1-m}a/e,bc/a,bd/a,be/a,q^{m}b/{\mu},q^{1-2m}{\mu};q)_{\infty}}\\[1mm]
&&\xqdn\:\times\:\:{_8\phi_7}\ffnk{ccccccc}
 {q,\frac{q^{1+m}}{f}}{q^{-2m}{\mu},q^{1-m}{\mu}^{\frac{1}{2}},-q^{1-m}{\mu}^{\frac{1}{2}},q^{-m}{\mu} f/a,q^{-m}b,q^{-m}c,q^{-m}d,q^{-m}e}
  {q^{-m}{\mu}^{\frac{1}{2}},-q^{-m}{\mu}^{\frac{1}{2}},q^{1-m}a/f,q^{1-m}{\mu}/b,q^{1-m}{\mu}/c,q^{1-m}{\mu}/d,q^{1-m}{\mu}/e}
  \nnm\\[1mm]\nnm
 &&\xqdn\:\times\:\:
\frac{(q^{1-2m}a,q^{1+m}b/a,bq/c,bq/d,bq/e,bq/f,q^{-m}c;q)_{\infty}}
  {(q^{1-m}a/b,q^{1-m}a/c,q^{1-m}a/d,q^{1-m}a/e,q^{1-m}a/f,bc/a,bd/a;q)_{\infty}}\\[1mm]
  &&\xqdn\:\times\:\:\frac{(q^{-m}d,q^{-m}e,aq/bf,bcde/a^2q,a^2q^2/bcde;q)_{\infty}}{(be/a,q^{-m-1}cde/a,q^{2+m}a/cde,q^{1+m}/f,b^2q/a;q)_{\infty}}
  \nnm\\[1mm]
 &&\xqdn\:\times\:\:{_8\phi_7}\ffnk{ccccccc}
 {q,\frac{q^{2+m}a^2}{bcdef}}{b^2/a,qba^{-\frac{1}{2}},-qba^{-\frac{1}{2}},q^{-m}b,bc/a,bd/a,be/a,bf/a}
  {ba^{-\frac{1}{2}},-ba^{-\frac{1}{2}},q^{1+m}b/a,bq/c,bq/d,bq/e,bq/f}.
 \enm
 By means of Lemma \ref{lemma-add}, the last transformation can be
 manipulated as
 \bnm
&&\qdn\xqdn{_8\psi_8}\ffnk{ccccccc}
 {q,\frac{q^{2+m}a^2}{bcdef}}{q^{-m}a,qa^{\frac{1}{2}},-qa^{\frac{1}{2}},b,c,d,e,f}{q^{1+m},a^{\frac{1}{2}},-a^{\frac{1}{2}},aq/b,aq/c,aq/d,aq/e,aq/f}
  \\ [1mm]
  &&\qdn\xqdn\:=\frac{(aq,q/a,aq/cd,aq/ce,aq/de,q^{1+m}/f,b/a,{\mu}q/c,{\mu}q/d,{\mu}q/e,aq/{\mu}f,q^{1+m}/{\mu};q)_{\infty}}
  {(q/f,q^{1+m}/a,aq/c,aq/d,aq/e,bc/a,bd/a,be/a,b/{\mu},{\mu}q,q/{\mu},q^{1+m}a/{\mu}f;q)_{\infty}}\\[1mm]
  &&\qdn\xqdn\:\times\: {_8\psi_8}\ffnk{cccccccc}
 {q;\frac{q^{1+m}}{f}}{q^{-m}{\mu},q{\mu}^{\frac{1}{2}},-q{\mu}^{\frac{1}{2}},b,c,d,e,{\mu}f/a}{q^{1+m},{\mu}^{\frac{1}{2}},-{\mu}^{\frac{1}{2}},{\mu}q/b,{\mu}q/c,{\mu}q/d,{\mu}q/e,aq/f}
 \\[1mm]
  &&\qdn\xqdn\:+\:\frac{(q,qa,q/a,c,d,e,bq/c,bq/d,bq/e,bq/f,q^{1+m}b/a;q)_{\infty}}
  {(q/f,q^{1+m}/a,aq/b,aq/c,aq/d,aq/e,aq/f,q^{1+m},bc/a,bd/a,be/a;q)_{\infty}}\\[1mm]
  &&\qdn\xqdn\:\times\:\frac{(aq/bf,q^{1+m}/b,bcde/a^2q,a^2q^2/bcde;q)_{\infty}}{(q/b,b^2q/a,cde/aq,aq^2/cde;q)_{\infty}}
  \\[1mm]
 &&\qdn\xqdn\:\times\:{_8\phi_7}\ffnk{ccccccc}
 {q,\frac{q^{2+m}a^2}{bcdef}}{b^2/a,qba^{-\frac{1}{2}},-qba^{-\frac{1}{2}},q^{-m}b,bc/a,bd/a,be/a,bf/a}
  {ba^{-\frac{1}{2}},-ba^{-\frac{1}{2}},q^{1+m}b/a,bq/c,bq/d,bq/e,bq/f}.
 \enm
This indicates that Theorem \ref{thm-cc} is true for $aq/g=q^{1+m}$.
Thus we have proved Theorem \ref{thm-cc} through Lemma \ref{lemma}
and the analytic continuation argument.
\end{proof}

 \textbf{Acknowledgments}

 The work is supported by the National Natural Science Foundations of China (Nos. 12071103 and 11661032).



\end{document}